\documentclass[11pt]{article}
\usepackage{fancyhdr}
\usepackage{fancyvrb}
\usepackage{tikz}
\usepackage{pgfplots}
\usepackage{amsmath}
\usepackage{flexisym}
\usepackage{breqn}
\usepackage{amssymb}
\usepackage{amsthm}
\usepackage{bbm}
\usepackage{url}
\usepackage{enumerate}
\usepackage{hyperref}
\usepackage{mathrsfs}  
\usepackage[normalem]{ulem}
\usepackage[style = numeric, sorting=nyt, doi=false,isbn=false,url = false, maxbibnames = 99]{biblatex}
\renewbibmacro{in:}{}
\pgfplotsset{compat=1.11}
\usepackage{graphicx}
\newtheorem{Theorem}{Theorem}
\newtheorem{definition}[Theorem]{Definition}

\newtheorem{remark}[Theorem]{Remark}

\providecommand{\keywords}[1]
{
  \small	
  \textbf{\textit{Keywords---}} #1
}
\numberwithin{equation}{section}
\numberwithin{Theorem}{section}
\DeclareMathOperator*{\esssup}{ess\,sup}

\DeclareMathOperator*{\cadlag}{\mathit{D}([0, T], \mathbb{R}^d)}
\DeclareMathOperator*{\R}{\mathbb{R}}

\DeclareMathOperator*{\DifferentialSetLeft}{\mathscr{C}^{i, j}_l([0, T], \R^d)}
\DeclareMathOperator*{\DifferentialSet}{\mathscr{C}^{i, j}([0, T], \R^d)}
\DeclareMathOperator*{\xproc}{\mathit{(X(t))}_{t \in [0, T]}}

\begin{document}
\title{An Optimal Functional It\^{o}'s Formula For L\'{e}vy Processes}
\author{Christian Houdr\'{e}\footnote{School of Mathematics, Georgia Institute of Technology, Atlanta, GA, 30332-0160, USA, houdre@math.gatech.edu. Research supported in part by the Simons Foundation grant \#524678.}\,\, and Jorge V\'{i}quez\footnote{Research done as a student in School of Mathematics, Georgia Institute of Technology, Atlanta, GA, 30332-0160, USA, jviquez6@gatech.edu.\
Research supported in part by the Department of Mathematics of the University of Costa Rica.}}

\maketitle

\begin{abstract}
Several versions of It\^{o}'s formula have been obtained in the setting of the functional stochastic calculus. In this regard, we present a local time-space version that works for arbitrary bounded and continuous functionals of L\'{e}vy processes and which does not depend on a functional's H\"{o}lder continuity.
\end{abstract}

\\


\section{Introduction, Notations and Definitions}

Dupire \cite{Dupire09} introduced notions of vertical and horizontal derivatives leading to a version of It\^{o}'s formula,  which has proved useful in financial applications and has seen  further generalizations, e.g., see \cite{CONT20101043}, \cite{Cont13}, \cite{LEVENTAL13}, \cite{OberhauserStratonovich}, \cite{Oberhauser}, \cite{Saporito}, \cite{Bouchard}, \cite{bouchardOctober}.  
With a different approach, \cite{FunctionalLocalTime} extended ideas from \cite{2DYoung} to obtain a corresponding formula for functionals with weak vertical derivatives of bounded $(p, q)$-variation.  
Below, instead, we develop in the functional setting Eisenbaum's,   \cite{Eisenbaum00}, \cite{EISENBAUM}, \cite{ReversibleLocalTime}, \cite{EisenbaumParabolic}, \cite{Eisenbaum09}, local time-space approach to It\^{o}'s formula, and obtain an optimal functional It\^{o}'s formula for multivariate L\'{e}vy processes.

The notations used in the present manuscript are the same as those  used in \cite{thesis}, some of them come from \cite{Dupire09}, some from \cite{Cont13, bouchardOctober}, and some are new.  We work on  the space of continuous functions $C([0, T], \mathbb{R}^d)$ and on the space of c\`{a}dl\`{a}g functions $D([0, T], \mathbb{R}^d)$ with domain $[0, T]$ and co-domain $\mathbb{R}^d$; defining the canonical random process $(X(t))_{t \in [0, T]}$ on the latter.  We also work on a filtered probability space $(D([0, T], \mathbb{R}^d), \mathcal{F}, (\mathcal{F}_t)_{t \in [0, T]}, \mathbb{P})$, where  $\mathcal{F}$ is a $\sigma$-field with $\sigma\{X(t):t \in [0, T]\} \subseteq \mathcal{F}$, $(\mathcal{F}_t)_{t \in [0, T]}$ is 
a right-continuous filtration completed with respect to $\mathbb{P}$, and $(X(t))_{t \in [0, T]}$ is progressively adapted to the filtration.  Finally, $\mathbb{P}$ is a probability measure for which $(X(t))_{t \in [0, T]}$ is a c\`{a}dl\`{a}g semi-martingale.

Next, for $w \in D([0, T], \mathbb{R}^d)$, let $w_{\wedge t}$ denotes a path that has been stopped at $t \in [0,T]$, and horizontally extended, i.e., 

\begin{align}
    w_{\wedge t}(s) &:= 
    \begin{cases} 
    w(s), & \text{if } s < t. \\
    w(t), & \text{if } s \geq t;  
    \end{cases}\nonumber
\end{align}
while the spatial perturbation of $w$ at $t\in [0,T]$ in the direction of $h \in \mathbb{R}^d$, is:

\begin{align}
    w_{\wedge t}^h(s) &:= 
    \begin{cases} 
    w(s), & \text{if } s < t. \\
    w(t)+h, & \text{if } s \geq t. 
    \end{cases}\nonumber
\end{align}

Throughout, functionals $F:[0,T]\times D([0,T], \mathbb{R}^d) \to \mathbb{R}$ are assumed to be measurable with respect to $\mathcal{B}([0,T])\otimes \mathcal{F}$, where $\mathcal{B}([0,T])$ is the Borel $\sigma$-field of $[0,T]$. Further, to describe the regularity properties of these functionals, $[0,T]\times D([0,T], \mathbb{R}^d)$ is equipped with a pseudo-metric $d_*$: 

\begin{align}
    d_*((t,w), (s,v)) := |t-s|+d_D(w_{\wedge t},v_{\wedge s}),\nonumber 
\end{align}
where $d_D$ is any metric on $D([0,T],\mathbb{R}^d)$ weaker than the one induced by the uniform norm; and thus includes the metrics associated with the Skorokhod topology. 

The change of variables formula we study describes the functional's sensitivity to horizontal extensions, where the last value is kept constant, and shocks on the last observed value of the process. With this in mind, we start with Dupire's definitions of derivatives in time and space. 

\begin{definition}\label{Horizontal Derivative}
A functional $F$ is horizontally-(time)-differentiable at 
$(t,x) \in [0,T)\times D([0,T],\mathbb{R}^d)$, if the following limit exists:
\begin{align}\label{eq:horizontal_derivative}
    DF(t,x_{\wedge t}) &:= \lim_{h\to 0^+} \frac{F(t+h,x_{\wedge t})-F(t,x_{\wedge t})}{h}.
\end{align}
\end{definition}
Thus, if the limit in \eqref{eq:horizontal_derivative} exists for any pair $(t,x)$, the functional $DF: [0, T)\times D([0, T],\mathbb{R}^d) \to \mathbb{R}$ associating $(t,x)$ to its horizontal derivative is well defined. In addition, \cite{Dupire09} defines the space derivative as:  
\begin{definition}
    A functional $F$ is space-differentiable at $(t,x) \in [0,T]\times D([0, T],\mathbb{R}^d)$, in the direction of the canonical basis vector $e_i$, $i \in \{1,...,d\}$, if the following limit exists:
\begin{align}\label{space_derivative}
    \partial_i F(t,x_{\wedge t}) &:= \lim_{h \to 0}\frac{F(t,x_{\wedge t}^{he_i})-F(t,x_{\wedge t})}{h}.
\end{align}
\end{definition}
\noindent 
Thus, if the limit in \eqref{space_derivative} exists for any pair $(t,x)$, the functional $\partial_i F: [0, T]\times D([0, T],\mathbb{R}^d) \to \R$ associating a pair to its space derivative is well defined.  Moreover, if $F$ is differentiable in the direction of all the canonical basis vectors $e_i, i =1,...,d$, $F$ is called space differentiable and its gradient is: 
\begin{align}
    \nabla F(t,x_{\wedge t}) &:= (\partial_1 F(t,x_{\wedge t}), \partial_2 F(t,x_{\wedge t}),...,\partial_d F(t,x_{\wedge t})).\nonumber
\end{align}

\noindent Finally, if each $\partial_i F$ is differentiable in the direction of all the canonical basis vectors $e_j: j = 1,2,...,d$, the Hessian $\nabla^2$ is the matrix such that $(\nabla^2 F(t,x_{\wedge t}))_{1\le i,j\le d} = (\partial_{i,j} F(t,x_{\wedge t}))_{1\le i,j\le d}$ under conditions ensuring that  $(\partial_i(\partial_j F(t,x_{\wedge t})))_{1\le i,j\le d} = (\partial_j(\partial_i F(t,x_{\wedge t})))_{1\le i,j\le d}$.\\

Recall that a functional $F$ is \emph{non-anticipative} if for all $(t,x) \in [0,T]\times D([0,T],\mathbb{R}^d)$,
    \begin{align}
        F(t,x) &= F(t,x_{\wedge t}).\nonumber
    \end{align}
If $F$ is a non-anticipative functional, and $x,y \in [0,T]\times D([0,T],\mathbb{R}^d)$ are such that, for $t \in [0,T]$, $x_{\wedge t} = y_{\wedge t}$, then $F(t,x) = F(t,y)$.

Note that the definitions of horizontal (time) and space differentiability still apply to anticipative functionals. However, when defining the space derivative, the actual values of the function $x$ should be taken instead of an horizontal extension past time $t$. For example, if $s > 0, (t,x) \in [0,T]\times D([0,T],\mathbb{R})$, and $F(t,x) = x(t+s)$. Then,
\begin{align*}
    \nabla F(t,x) &= \lim_{h \to 0} \frac{x(t+s)+h-x(t+s)}{h}= 1.
\end{align*}

\noindent \emph{Throughout, all functionals will be taken to be non-anticipative.}

Next, notions of regularity for the functionals and their derivatives are recalled, e.g., \cite{Cont13}.

\begin{definition}
    A functional $F$ is said to be fixed-time continuous at $(t,x) \in [0,T]\times D([0,T],\mathbb{R}^d)$ if it is continuous after fixing the time variable $t$, i.e., if for all $\epsilon > 0$, there exists a $\delta > 0$ such that for all $y\in D([0,T],\mathbb{R}^d)$,
    \begin{align}
        d_D(x_{\wedge t}, y_{\wedge t}) < \delta \implies |F(t,x_{\wedge t})-F(t,y_{\wedge t})| < \epsilon.\nonumber
    \end{align}
\end{definition}

\noindent If $F$ is fixed-time continuous at all $(t,x) \in [0,T]\times D([0,T],\mathbb{R}^d)$, then $F$ is said to be fixed-time continuous. 

\begin{definition}
    A functional $F$ is said to be left-continuous (in time) at $(t,x) \in [0,T]\times D([0,T],\mathbb{R}^d)$ if for all $\epsilon > 0$, there exists $\delta > 0$ such that for all $y \in \cadlag$,
    \begin{align}
        d_*((t,x_{\wedge t}), (s, y_{\wedge s})) <\delta,\, s\leq t \implies |F(t,x_{\wedge t})-F(s,y_{\wedge s})|< \epsilon.\nonumber
    \end{align}
\end{definition}
    
\noindent A functional that is left-continuous for all $(t,x)\in [0, T]\times $ is said to be left-continuous in time and following the notation of \cite{Cont13}, the set of such functionals is denoted by $\mathscr{C}_l([0,T],\mathbb{R}^d)$. Right-continuity in time is defined analogously, and a functional is said to be continuous in time if it is both left and right continuous in time. The set of such functionals is denoted by $\mathscr{C}([0,T],\mathbb{R}^d)$.

The following is an analog of locally bounded functions,
    \begin{definition}
        A functional $F:[0,T]\times D([0,T],\mathbb{R}^d)\to \mathbb{R}$ is said to be boundedness-preserving if for every $t \in [0,T)$ and compact set $K \subset \R^d$, there exists a constant $C_{t,K}$ such that:
        \begin{align}
            x(s) \in K,\,\, \forall s \in [0,t] \implies |F(s,x_{\wedge s})| \leq C_{t,K},\,\,\forall s \in [0,t].\nonumber
        \end{align}
    \end{definition}
\noindent The set of boundedness-preserving functionals is denoted $\mathscr{C}_b([0,T],\mathbb{R}^d)$
    
    Throughout, let $\mathscr{C}^{0,0}_l([0,T],\mathbb{R}^d) := \mathscr{C}_b([0,T],\mathbb{R}^d)\cap \mathscr{C}_l([0,T],\mathbb{R}^d)$, and for $i,j$ not both $0$, let $\DifferentialSetLeft$ denote the set of functionals $F:[0, T]\times D([0,T],\mathbb{R}^d)\to \R$ that are $i$-times  horizontally differentiable and $j$-times space differentiable, such that $F$ and all its space derivatives are in $\mathscr{C}^{0,0}_l([0,T],\mathbb{R}^d)$, while the horizontal derivatives are in $\mathscr{C}_b([0,T],\mathbb{R}^d)$ and are continuous at fixed-times.
         
    In an analogous manner $\mathscr{C}^{0,0}([0,T],\mathbb{R}^d) := \mathscr{C}_b([0,T],\mathbb{R}^d)\cap \mathscr{C}([0,T],\mathbb{R}^d)$ and for $i,j$ not both $0$, let $\DifferentialSet$ denote the set of functionals that are $i$-times  horizontally differentiable and $j$-times space differentiable, such that $F$ and all its space derivatives are in $\mathscr{C}^{0,0}([0,T],\mathbb{R}^d)$, while the horizontal derivatives are in $\mathscr{C}_b([0,T],\mathbb{R}^d)$ and are continuous at fixed-times.
    
    With these notations, the functional It\^{o}'s formula established in \cite{Dupire09} and generalized in \cite{CONT20101043, Cont13} becomes:   

    \begin{Theorem}\label{Functional_Ito_Formula}
    Let $F \in \mathscr{C}^{1,2}_l([0,T],\mathbb{R}^d)$, let $X$ be a continuous semi-martingale under the probability space $(\Omega,\mathcal{F},(\mathcal{F})_{t \in [0, T]}, \mathbb{P})$, and let $([X](t))_{t \in [0, T]}$ be its quadratic covariation. Then,
    \begin{align}\label{eq:FunctionalIto}
        F(t,X_{\wedge t})-F(0,X_{\wedge 0}) = \int_0^t DF(s,X_{\wedge s})\,\mathrm{d}s&+\int_0^t \nabla F(s,X_{\wedge s})\cdot \mathrm{d}X(s)
     +\frac{1}{2}\int_0^t Tr(\nabla^2 F(s,X_{\wedge s})\mathrm{d}[X](s)),\text{a.s.}
    \end{align}
    \end{Theorem}

\section{An Optimal Functional It\^{o}'s Formula}

The main objective below is to relax the space differentiability conditions on $F$ and the convergence requirements in the last integral in \eqref{eq:FunctionalIto}, by developing functional analogs of the local time-space arguments of \cite{Eisenbaum09}.  
A first step is to note that the definition of the integral with respect to local time found in \cite[Section 2]{Eisenbaum09} continues to hold when an additional process $Y \in \cadlag$, independent of the Brownian component of $(X(t))_{t\in [0,T]}$, is added to the functional. This fact allows for the use of the martingale properties of the Brownian motion to separate the present from the past and the process' jumps while studying the increments of $F(t,X_{\wedge t})$, and thus obtain the following:  \\

\begin{Theorem}\label{Integral with respect to local time}
Let $F:[0,T]\times \cadlag \times \mathbb{R}^d \to \mathbb{R}$ be boundedness preserving, let $(B(t))_{t \in [0,T]}$ be a $d$-dimensional standard Brownian motion adapted to the filtration $(\mathcal{F}_t)_{t \in [0, T]}$, and let $(Y(t))_{t\in [0,T]}$ be a progressively measurable c\`{a}dl\`{a}g process, such that, $\mathcal{Y} = \sigma\{Y(t):t \in [0,T]\}$ is independent of $\sigma\{B(t):t \in [0,T]\}$. Then, the integral with respect to the Brownian motion's local time measure $dL^x_s(B)$ satisfies for $j=1, \dots, d$:
\begin{align}
    \int_0^t\int_{\mathbb{R}}F(s,Y_{\wedge s^-},B(s)|_{B^j(s) = x})\,\mathrm{d}L^x_s(B^j) = &\int_0^t\!\!\! F(s,Y_{\wedge s^-},B(s))\,\mathrm{d}B^j(s)+\int_{T-t}^T\!\!\!\!\!\!F(T-s,Y_{\wedge (T-s)^-}, \hat{B}(s))\,\mathrm{d}\hat{B}^j(s),\nonumber
\end{align}
\noindent where for $x = (x^1,...,x^d) \in \mathbb{R}^d$, $x|_{x^j = y} = (x^1,...,x^{j-1},y,x^{j+1},...,x^d)$, and 
$\hat{x}(t) = (\hat{x}^1(t),...,\hat{x}^d(t))$ $ = (x^1((T-t)^-),...,x^d((T-t)^-))$.
\end{Theorem}

As in \cite{Eisenbaum09}, the integral with respect to the Brownian local time $L^x_s(B)$ is first defined for simple functionals, i.e., for those functionals $F_\Delta$ for which there exists a partition $\Delta := \{s_1 < s_2 < ... < s_n\}\times\{x_1 < x_2 <...< x_m\}$ such that  
\begin{align}
    F_\Delta(t, x_{\wedge t}, x) &= \sum_{(s_i, x_j) \in \Delta} F_{i,j}(w_{\wedge s_i})\mathbbm{1}_{(s_i, s_{i+1}]}(t)\mathbbm{1}_{(x_j, x_{j+1}]}(x).\nonumber
\end{align}
For these functionals, the integral is equal to,

 \begin{align}\label{int_simple}
 \int_0^t \int_{\mathbb{R}} F_{\Delta}(t,w_{\wedge t}, x)\mathrm{d}L^x_t(B) &= \sum_{(s_i,x_j)\in \Delta} (F_{i,j}(w_{\wedge s_i})(L^{x_{j+1}}_{s_{i+1}\wedge t}(B)-L^{x_{j+1}}_{s_i\wedge t}(B)-L^{x_{j}}_{s_{i+1}\wedge t}(B)+L^{x_{j}}_{s_i\wedge t}(B)))\nonumber\\
 &=\int_0^tF_{\Delta}(s,w_{\wedge s},B(s))\mathrm{d}B(s)+\int_{T-t}^T F_\Delta(T-s,w_{\wedge (T-s)}, \hat{B}(s))\mathrm{d}\hat{B}(s).\end{align}

The second equality in \eqref{int_simple} follows from the corresponding result in \cite[Theorem 2.1]{EISENBAUM}. Then, Theorem \ref{Integral with respect to local time} and its proof follow from the corresponding deterministic integrand results in \cite[Theorem 2.1]{Eisenbaum09}, the fact that $(Y(t))_{t \in [0,T]}$ and $(B(t))_{t \in [0,T]}$ are independent, and \cite[Chapter IV. Theorem 65]{Protter}. Moreover, for boundedness-preserving functionals $F$, the integral with respect to local time satisfies, for $j \in \{1,2,...,d\}$, 

\begin{align}
    \mathbb{E}\left|\int_0^t\int_{\mathbb{R}}F(s,Y_{\wedge s^-},B(s)|_{B^j(s)=x})\mathrm{d}L_s^x(B^j)\right| \leq  \|F\|_L,\nonumber
\end{align}
where for $F:[0,T]\times \cadlag\times \mathbb{R}^d\to \mathbb{R}^d$ such that for all $x\in\mathbb{R}^d$, $F(\cdot,x)$ is boundedness preserving, and for all $(s,w)\in [0,T]\times \cadlag$, $F(s,w,\cdot)$ is locally bounded, and Borel measurable, 
    \begin{align}\label{local_time_norm_eq}
    \|F\|_L &:= 2\left(\mathbb{E}\left(\int_0^T F^2(t,Y_{\wedge t}, B(t))\mathrm{d}t\right)\right)^{1/2}+\mathbb{E}\left(\int_0^T |F(t,Y_{\wedge t}, B(t))|\frac{\|B(t)\|_1}{t}\mathrm{d}t\right), 
\end{align}
    \noindent
where, $\|\cdot\|_1$ denotes the Euclidean $\ell_1$-norm.
Using the local time norm $\|\cdot\|_L$ defined above, the integral can 
then be extended to $\mathscr{C}^{0,0}([0,T],\mathbb{R}^d)$ through density arguments.

Next, and from its validity for deterministic integrands  \cite{EISENBAUM}, for boundedness-preserving functionals $F:[0,T]\times \cadlag\times \mathbb{R}^d \to \mathbb{R}$ with a gradient of weak derivatives $\nabla F = (\partial_1 F,..., \partial_d F)$, the identity
    \begin{align}
        \int_0^T\int_{\mathbb{R}} F(t,Y_{\wedge t^-}, B(t)|_{B^j(t) = x})\mathrm{d}L^x_t(B^j) = -\int_0^T \partial_j F(t,Y_{\wedge t^-},B(t))\mathrm{d}t,\nonumber
    \end{align}
    \noindent is satisfied for all $j \in \{1,2,...,d\}$.

The main setting in which Theorem~\ref{Integral with respect to local time} is applied is that of a L\'{e}vy process $(X(t))_{t\in [0,T]}$, with characteristic triplet $(\mu,\Sigma,\nu)$. Then, classically, $(X(t))_{t\in [0,T]}$ has a L\'{e}vy-It\^{o} decomposition, and so for all $t\in [0,T]$, $X(t) = \mu t+ \Sigma^{1/2}B(t)+N(t)$,
where $(B(t))_{t \in [0, T]}$ and $(N(t))_{t \in [0,T]}$ are respectively a multidimensional Brownian motion and a pure jump L\'{e}vy process with characteristic triplet $(0,0,\nu)$ and compensated measure $\tilde{N}$, both progressively adapted to the filtration generated by $(X(t))_{t\in [0,T]}$.

The two main choices for the process $(Y(t))_{t \in [0,T]}$ in Theorem~\ref{Integral with respect to local time} are $Y(t) = N(t)$, when proving the optimal It\^{o} formula for multivariate L\'{e}vy processes, or $Y(t) = Z_{\wedge s}(t)+N(t)$ for some $s \in [0,t]$ 
where $(Z(t))_{t \in [0,T]}$ is a progressively measurable process adapted to $(\mathcal{F}_t)_{t \in [0, T]}$. In this last instance, using the Brownian motion $B_s(t) := B(t)-B(s)$ for $t \in [s,T]$, gives the following version of Theorem~\ref{Integral with respect to local time}.

\begin{Theorem}\label{Displaced_Integral}
    Let $F:[0,T]\times \cadlag \times \mathbb{R}^{2d} \to \mathbb{R}$ be such that for all $(t,w) \in [0,T]\times \cadlag$, $F(t,w_{\wedge t},\cdot,\cdot)$ is locally bounded. Let $(B(t))_{t \in [0,T]}$ be a standard $d$-dimensional Brownian motion adapted to the filtration $(\mathcal{F}_t)_{t \in [0, T]}$, let $(B_s(t)=B(t)-B(s))_{t \in [s,T]}$, and let $(Z(t))_{t \in [0,T]}$ a progressively measurable process adapted to $(\mathcal{F}_t)_{t \in [0, T]}$. Then, the integral with respect to the Brownian motion's local time measure $dL^x_s(B_s)$ satisfies for all $j=1, \dots, d$:
\begin{align}
    \int_s^t\int_{\mathbb{R}}F(s, Z_{\wedge s},N(r), B_s(r)|_{B^j_s(r) = x})\,\mathrm{d}L^x_r(B^j_s)= \int_s^t\int_{\mathbb{R}}F(s, Z_{\wedge s},N(r), B_s(r)|_{B^j_s(r) = x-B(s)})\,\mathrm{d}L^x_r(B^j)\nonumber&\\
    =\int_s^t F(s,Z_{\wedge s},N(r),B_s(r))\,\mathrm{d}B^j(r)+\int_{T-t}^{T-s} F(s,Z_{\wedge{s}},N(T-r),\hat{B}_s(r))\,\mathrm{d}\hat{B}^j(r)&.\nonumber
\end{align}
\end{Theorem}

Since $(N(t^-))_{t \in [s,T]}$ is also independent of $\sigma\{B(t): t \in [s,T]\}$, and since it only has a countable number of jumps, an approximation in the local time norm, allows to see that,
\begin{align}
     \int_{T-t}^{T-s} F(s, Z_{\wedge s},N(T-r),\hat{B}(r))\mathrm{d}\hat{B}(r) = \int_{T-t}^{T-s} F(s, Z_{\wedge s},\hat{N}(r),\hat{B}(r))\mathrm{d}\hat{B}(r).\nonumber
\end{align}
\noindent Therefore, Theorem~\ref{Displaced_Integral} can be written in a way more akin to the previous results in the literature, in particular, \cite[Theorem 5]{EisenbaumParabolic}. Having defined the notion of integration with respect to local time that is used in the rest of the paper, we first proceed to prove a version of \cite[Theorem 1.1]{Eisenbaum09} for functions of multivariate L\'{e}vy process with possibly singular covariance matrix $\Sigma$. However, the main proof idea in \cite{Eisenbaum09} is to use the Brownian motion's local time to rewrite the function to allow additional regularity when considering the effect of jumps. 

Now, for $i\geq 1$, let $C^i(\mathbb{R}^d,\mathbb{R})$ denote the set of $i$-times weakly differentiable functions from $\mathbb{R}^d$ to $\mathbb{R}$ such that the function and all its derivatives are locally bounded, while $C^0$ is the set of locally bounded functions. Then, define the following operators.

\begin{definition}
For $i \in \{1,...,d\}$, let ${\cal I}_i: C^0(\mathbb{R}^d,\mathbb{R}) \to C^0(\mathbb{R}^d,\mathbb{R})$, and let ${\cal A}_i: C^2(\mathbb{R}^d,\mathbb{R})\to C^0(\mathbb{R}^d,\mathbb{R})$, be defined via ${\cal I}_i := \int_0^{x^i} f(x|_{x^i = y})\,\mathrm{d}y$ and ${\cal A}_i :=  \frac{1}{2}\frac{\partial^2f}{\partial x_i^2} (x)+\int_{\| y\|_2 < 1}\int_0^1\left(\frac{\partial f}{\partial x_i}(x+uRy)-\frac{\partial f}{\partial x_i}(x)\right)(Ry)^i\mathrm{d}u\nu(\mathrm{d}y)$,\\
\noindent where $R$ a $d \times m$, $m \ge 1$ matrix.  
\end{definition}

\begin{Theorem}\label{multivariate_local_time}
Let $(X(t))_{t \in [0,T]}$ be a multivariate L\'{e}vy Process with triplet $(\mu,\Sigma,\nu)$, let $Q: \mathbb{R}^d\to \mathbb{R}^d$ be the orthogonal projection onto the range of $\Sigma^{1/2}$, let $f:\mathbb{R}^d \to \mathbb{R}$ be continuously differentiable, let $\tilde{f}:\mathbb{R}\times\mathbb{R}^m\to \mathbb{R}$ be such that $\tilde{f}(t,x) = f(\Sigma^{1/2}x+X^d(t^-))$, where $m$ is the rank of $\Sigma^{1/2}$, and let 
\begin{align}
    \int_{\| x\|_2 < 1} \| (I_d-Q) x\|_2\,\nu(\mathrm{d}x) < \infty,\nonumber
\end{align} 
\noindent where $\|\cdot\|_2$ is the Euclidean norm.
\noindent Then,
\begin{align}\label{Local_Time_Ito}
   f(X(t))-f(0) &= \int_0^t \nabla f(X(s^-))^T \Sigma^{1/2}\cdot\mathrm{d}B(s)+\int_0^t \langle \nabla f(X(s^-)),\mu\rangle\,\mathrm{d}s\nonumber\\
   &+\int_0^t\int_{\| y\|_2 < 1} \!\!\!\!\!\!\!\!\!\!\left(f(X(s^-)+y)-f(X(s^-))\right)\,\tilde{N}(\mathrm{d}s,\mathrm{d}y)\nonumber\\
    &-\sum_{i = 1}^m \int_0^t \int_{\mathbb{R}} \mathcal{A}_i\mathcal{I}_i\tilde{f}(s,B(s)|_{B^i(s) = x})\,\mathrm{d}L^x_s(B^i)+\sum_{s \in [0,t]} \left(f(X(t))-f(X(t^-))\right)\mathbbm{1}_{\{\|\Delta X(s)\|_2 \geq 1\}}\nonumber\\
    &+\int_0^t\int_{\| y\|_2 < 1}\!\!\!\!\!\!\!\!\!\!\left(f(X(s^-)+y)-f(X(s^-)+Qy)-\left\langle\nabla f(X(s^-)),(I_d-Q)y\right\rangle\right)\,\nu(\mathrm{d}y)\mathrm{d}s,
\end{align}
\noindent where the matrix $R$ in $\mathcal{A}_i$'s definition is chosen as the left-inverse of the matrix associated to $Q$, i.e. $R:= ((\Sigma^{1/2})^T \Sigma^{1/2})^{-1}(\Sigma^{1/2})^T$.
\end{Theorem}

The proof of Theorem \ref{multivariate_local_time} follows the steps outlined in \cite{Eisenbaum09}, modified to account for the case when $\Sigma$ is not of full rank. The interested reader can find the details in \cite{thesis}.

The next theorem is the main result of the paper providing a functional analogue of the previous one (and in fact is equivalent to it).

\begin{Theorem}\label{Optimal_Functional}

Let $F \in \mathscr{C}^{1,1}([0,T],\mathbb{R}^d)$, and let $ X = (X(t))_{t \in [0,T]}$ be a L\'{e}vy process with triplet $(\mu,\Sigma,\nu)$. Let, $Q$, the orthogonal projection onto the range of $\Sigma^{1/2}$, be such that $\int_{\| y\|_2 < 1}\| (I_d-Q)y\|_2\,\nu(\mathrm{d}y) < \infty$. Then,
\begin{align}
\label{Optimal_Ito_Formula}
F(t,X_{\wedge t})-F(0,X_{\wedge 0}) &= \int_0^t DF(s,X_{\wedge s})\,\mathrm{d}s+\int_0^t \langle \nabla F(s,X_{\wedge s}),\mu\rangle\,\mathrm{d}s\nonumber\\
&+\int_0^t \nabla F(s,X_{\wedge s^-})^T\Sigma^{1/2}\,\mathrm{d}B(s)+ \sum_{s \leq t} (F(s,X_{\wedge s})-F(s,X_{\wedge s^-}))\mathbbm{1}_{\{\|\Delta X(s)\|_2 \geq 1\}}\nonumber\\
&+\int_0^t \int_{\|y\|_2 < 1} \left(F(s,X_{\wedge s^-}^y)-F(s,X_{\wedge s^-})\right)\,\tilde{N}(\mathrm{d}s,\mathrm{d}y)\nonumber\\
&-\sum_{j = 1}^m \int_0^t\int_{\mathbb{R}}\mathcal{A}_j\mathcal{I}_jF(s,X_{\wedge s^-}^{(\Sigma^{1/2}e_j) x-X(s^-)})\,\mathrm{d}L_s^x(B^j)\nonumber\\
&+\int_0^t\int_{\| y\|_2 < 1} \left(F(s,X_{\wedge s^-}^y)-F(s,X_{\wedge s^-}^{Qy})-\langle \nabla F(t,X_{\wedge t^-}),(I_d-Q)y\rangle\right)\,\nu(\mathrm{d}y)\,\mathrm{d}s,\nonumber\\&
\end{align}
\noindent with $\mathcal{A}_i$ defined as in Theorem \ref{multivariate_local_time}.
\end{Theorem}

\begin{proof}
Let $\tau = \{\tau_n\}_{n \geq 1}$ be a sequence of partitions given by stopping times $\tau_n = (t_0^n,...,t^n_{k_n})$, as in the proof of the functional It\^{o} formula in \cite{CONT20101043}. More precisely, define $t^n_0 = 0$, and 
\begin{align}
    &t^n_{i+1} = \inf\{t > t^n_i: \|\Delta X(t)\|_2 \geq 1/2^n\}\wedge (t^n_i+1/2^n)\wedge T\}\nonumber\\
    &k_n = \min\{n: t^n_i = T\}.\nonumber
\end{align}
\noindent Note that $X(t)-X(t^n_i)$ is independent of $\mathcal{F}_{t^n_i}$, and that the stopping times in the sequence of partitions $\{\tau_n\}$ are independent to the Brownian component of the process $\xproc$. Moreover, since $\xproc$ is c\`{a}dl\`{a}g, $k_n < \infty$.  Then, define $
X^n(t) := \sum_{i = 0}^{k_n-1} X(t_i^{n})\mathbbm{1}_{[t^n_i,t^n_{i+1})}(t)+X(T)\mathbbm{1}_{\{T\}}$, together with $
    F^n_i(x) = F\left(t^n_{i+1},(X^n_{\wedge t^{n-}_{i+1}})^{x-X^n(t^{n}_i)}\right)$. 
\noindent From the construction of $\tau$, $X^{n}(t)$ converges to $X(t^-)$, except at the jump times of $X$. However, as this set has Lebesgue measure $0$, then
\begin{align}
\esssup_{t \in [0,T]} \| X_n(t)-X(t)\|_2 \to 0.\nonumber    
\end{align}
\noindent Since,  $F(\cdot,X_{\wedge \cdot})$ is bounded, and since $F$ is continuous with respect to $d_*$, \cite[Appendix 1]{Dupire09} ensures that $\| F(t,X^n_{\wedge t})-F(t,X_{\wedge t})\|_\infty \to 0$. The same applies to the space derivatives of $F$, ensuring $\|\nabla F(t,X^n_{\wedge t})-\nabla F(t,X^n_{\wedge t})\|_\infty \to 0$. Next, from equation \eqref{Local_Time_Ito}:

\begin{align*}
    F(T,X^n_{\wedge T})-F(0,X^n(0))&= \sum_{i = 0}^{k_n-1} \left(F(t^n_{i+1},X^n_{\wedge t^n_{ i+1}})-F(t^n_{i+1},X^n_{\wedge t^{n-}_{i+1}})\right)+\sum_{i = 0}^{k_n-1} \left(F(t^n_{i+1},X^n_{\wedge t^{n-}_{i+1}})-F(t^n_i,X^n_{\wedge t^n_i})\right)\nonumber\\
    &= \sum_{i = 0}^{k_n-1} \left(F_i^n(X(t^n_{i+1}))-F_i^n(X(t^n_i))\right)+\int_{t^n_i}^{t^n_{i+1}}D F(t,X^n_{\wedge t_i})\,\mathrm{d}t\nonumber\\
    &= \sum_{i = 0}^{k_n-1} \!\left(\int_{t^n_i}^{t^n_{i+1}} \!\!\!\!\!\!\!\!DF(t,X_{\wedge t_i}^n)\,\mathrm{d}t+\!\!\int_{t^n_i}^{t^n_{i+1}}\!\!\!\!\!\!\!\! \nabla F_i^n(X( t^-))^T\Sigma^{1/2}\mathrm{d}B(t)\right)\nonumber\\
    &+\!\!\!\!\!\!\!\sum_{s \in (t^n_i,t^n_{i+1}]} \!\!\!\!\!\!\!\left(F^n_i(X^n(s))\!\!-\!\!F_i^n(X^n(s^-))\right)\mathbbm{1}_{\|\Delta x(s)\|_2 \geq 1}\nonumber\\
    &+\int_{(t^n_i,t^n_{i+1}]}\int_{\|y\|_2 < 1} \left(F_i^n(X^n(t^-)+x)-F_i^n(X^n(t^-))\right)\tilde{
    N}(\mathrm{d}t,\mathrm{d}x)\nonumber\\
    &-\sum_{j = 1}^m \int_0^t\int_{\mathbb{R}}\mathcal{A}_j\mathcal{I}_j\tilde{F}_i^n(B(s)|_{B^j(s) = x})\mathbbm{1}_{(t^n_i,t^n_{i+1}]}(s)\,\mathrm{d}L_s^x(B^j)\nonumber\\
    &= \int_0^T D F(t,X_{\wedge t}^n)\,\mathrm{d}t
    +\int_0^T \sum_{i = 0}^{k_n-1} \nabla F_i^n(X( t^-))^T\Sigma^{1/2}\mathbbm{1}_{(t^n_i,t^n_{i+1}]}(t)\cdot\mathrm{d}B(t)\nonumber\\
    &+ \int_0^T \int_{\|y\|_2 < 1} \sum_{i =0}^{k_n-1} (F_i^n(X^n(t^-)+y)-F_i^n(X^n(t^-)))\mathbbm{1}_{(t^n_i,t^n_{i+1}]}(t)\,\tilde{N}(\mathrm{d}t,\mathrm{d}y)\nonumber\\
    &+ \sum_{t \in [0,T]} \sum_{i = 0}^{k_n-1} (F^n_i(X^n(t)))-F_i^n(X^n(t^-)))\mathbbm{1}_{\|\Delta x(t)\|_2 \geq 1}\mathbbm{1}_{(t^n_i,t^n_{i+1}]}(t)\nonumber\\
    &-\sum_{j = 1}^m\int_0^t\int_{\mathbb{R}}\mathcal{A}_j\mathcal{I}_j\left(\sum_{i = 0}^{k_n-1}\tilde{F}_i^n(B(s)|_{B^j(s) = x})\mathbbm{1}_{(t^n_i,t^n_{i+1}]}(s)\right)\,\mathrm{d}L_s^x(B^j).\nonumber&
\end{align*}

The convergence of each of these terms is now verified. First, $DF$ is boundedness-preserving, and from the partition taken, $X^n(t) \to X(t)$, almost everywhere, thus by the Dominated Convergence Theorem, the first two integrals converge to $\int_0^T DF(t,X_{\wedge t})\,\mathrm{d}t+\int_0^T \langle \nabla F(t,X(t)),\mu\rangle\,\mathrm{d}t$.

Next, observe that the functional and its space derivatives are left-continuous in time and that $X(t^-)-X(t^n_i) \to 0$, uniformly for all $t \in (t^n_i,t^n_{i+1}]$, once again, from the choice of the partition. So, if $G$ is continuous, $\| G(t^n_{i+1},(X^n_{\wedge t_{i+1}^-})^{X(t^-)-X(t_i)}) - G(t,X_{\wedge t^-})\|_\infty\to 0$. This uniform convergence allows to replicate the argument in \cite{EISENBAUM} to obtain the convergence of the second integral to $\int_0^T \nabla F(t,X_{\wedge t^-})^T\Sigma^{1/2}\cdot\mathrm{d}B(t)$, using the Burkholder–Davis–Gundy’s inequality.\\

Since $\|\Delta X(\cdot)\|_2 \geq 1$ for finitely many $t$, the fourth integral converges to 
\begin{align*}
    \sum_{t \in [0,T]} (F(t,X_{\wedge t})-F(t,X_{\wedge t^-}))\mathbbm{1}_{\|\Delta X(t)\|_2 \geq 1}(t).
\end{align*}

\noindent Let us next analyze the integral with respect to the discontinuous martingale:
\begin{align*}
 \int_0^T \int_{\|y\|_2 < 1} \sum_{i =0}^{k_n-1} (F_i^n(X^n(t^-)+y)-F_i^n(X^n(t^-)))\mathbbm{1}_{(t^n_i,t^n_{i+1}]}(t)\,\tilde{N}(\mathrm{d}t,\mathrm{d}y)\\
 = \int_0^T \int_{\|y\|_2 < 1} \sum_{i =0}^{k_n-1} \int_0^1 \langle\nabla F^n_i(X^n(t^-)+hy),y\rangle\,\mathrm{d}h \mathbbm{1}_{(t^n_i,t^n_{i+1}]}(t)\,\tilde{N}(\mathrm{d}t,\mathrm{d}y)\\
 = \int_0^T \int_{\|y\|_2 < 1} \int_0^1 \sum_{i =0}^{k_n-1}\langle\nabla F^n_i(X^n(t^-)+hy),y\rangle\mathbbm{1}_{(t^n_i,t^n_{i+1}]}(t)\,\mathrm{d}h \,\tilde{N}(\mathrm{d}t,\mathrm{d}y).
\end{align*}
Once again, if $t \not \in \tau_n$ for any $n$, $\sum_{i =0}^{k_n}\langle\nabla F^n_i(X^n(t^-)+hy),y\rangle\mathbbm{1}_{(t^n_i,t^n_{i+1}]}(t)$ converges almost everywhere to $\langle\nabla F(t,X_{\wedge t^-}^{hy}),y\rangle$, and the convergence rate is uniform for $t \not \in \tau_n$. Thus, this convergence occurs almost everywhere in $[0,T]$. At first, let $X$ be a.s.~bounded by a constant $C > 1$, then $\|X(t^-)+hy\|_2 \leq 2C$, and:
\begin{align}
\mathbb{E}\Bigg(\int_0^T \int_{\|y\|_2 < 1}\int_0^1 \sum_{i = 0}^{k_n-1} \langle\nabla F^n_i(X^n(t^-)+hy),y\rangle\mathbbm{1}_{(t^n_i,t^n_{i+1}]}(t)-\nabla \langle F(t,X_{\wedge t^-}^{hy}),y\rangle\,\mathrm{d}h\tilde{N}(\mathrm{d}t,\mathrm{d}y)\Bigg)^2\nonumber&\\
= \mathbb{E}\int_0^T \int_{\{\|y\|_2 < 1\}}\Bigg(\int_0^1 \sum_{i = 0}^{k_n-1} \langle\nabla F^n_i(X^n(t^-)+hy),y\rangle\mathbbm{1}_{(t^n_i,t^n_{i+1}]}(t)-\nabla \langle F(t,X_{\wedge t^-}^{hy}),y\rangle\,\mathrm{d}h\Bigg)^2\nu(dy)\mathrm{dt}\nonumber&\\
\leq \sup_{t \in [0,T], \| x\|_\infty < C}\| \nabla F(t,x_{\wedge t})\|_2^2\,\,\int_0^T \int_{\|y\|_2 < 1} 4C^2y^2\nu(dy)dt < +\infty.\nonumber&
\end{align}
\noindent Thus, by stochastic dominated convergence, this integral converges in probability to \[\int_0^T\int_{\{\|y\|_2 < 1\}} \left(F(t,X_{\wedge t^-}^y)-F(t,X_{\wedge t^-})\right)\tilde{N}(\mathrm{d}t,\mathrm{d}y).\]
\noindent For the first term from the very definition of the $\mathcal{A}_j$ operators,
\begin{align*}
 \sum_{i = 0}^{k_n-1}\int_0^T\int_{\mathbb{R}}\frac{\partial^2}{\partial x_j^2} I_j\tilde{F}_i^n(B(s)|_{B^j(s) = x})\mathbbm{1}_{(t^n_i,t^n_{i+1}]}(s)\,\mathrm{d}L_s^x(B^j)
= \int_0^T\int_{\mathbb{R}}\sum_{i =0}^{k_n-1} \frac{\partial \tilde{F}_i^n}{\partial x^j}(B(s)|_{B^j(s) = x})\mathbbm{1}_{(t^n_i,t^n_{i+1}]}(s)\,\mathrm{d}L_s^x(B^j).   
\end{align*}
\noindent Then, 
\begin{align}
\label{convergence_derivative}
    \mathbb{E}\left\vert\int_0^T\int_{\mathbb{R}}\sum_{i = 0}^{k_n-1}\frac{d\tilde{F}_i^n(B(s)|_{B^j(s) = x})}{dx_j}\mathbbm{1}_{(t^n_i,t^n_{i+1}]}(s)-\langle\Sigma^{1/2}e_j,\nabla F(s,X_{\wedge s^-}^{(\Sigma^{1/2}e_j)x-X^j(s^-)})\rangle\,\mathrm{d}L^x_s\right\vert\nonumber\nonumber\\
    \leq \left\| \sum_{i = 0}^{k_n-1}\frac{\partial \tilde{F}_i^n(x)}{dx_j}\mathbbm{1}_{(t^n_i,t^n_{i+1}]}(s)-\langle\Sigma^{1/2}e_j,\nabla F(s,X_{\wedge s^-}^{x-X(s^-)})\rangle\right\|_L.
\end{align}
\noindent Since $\int_0^T \frac{\|B(t)\|_1}{t}\mathrm{d}t < \infty$, there exists $C > 0$ such that for any $F:[0,T]\times\cadlag\times \mathbb{R}^d \to \mathbb{R}$,\\ $\| F\|_{L} \leq C\| F\|_\infty$, and since in the infinity norm \[\sum_{i = 0}^{k_n-1} \frac{\partial \tilde{F}_i^n(x)}{dx_j}\mathbbm{1}_{(t^n_i,t^n_{i+1}]}(t) \longrightarrow \langle \Sigma^{1/2}e_j,\nabla F(s,X_{\wedge s^-}^{x-X(s^-)})\rangle,\]
the right-hand side of \eqref{convergence_derivative} converges to 0, obtaining convergence to the desired integral in \eqref{Optimal_Ito_Formula}. For the second integral in the $\mathcal A_j$ operator, define: 
\begin{align}
    H^j_{n}(t,x) &:= \int_{\| y\|_2\leq 1}\int_0^1 \sum_{i = 0}^{k_n-1}\Big(\tilde{F}^n_i(x+uRy)-\tilde{F}^n_i(x)\Big)(Ry)^j\mathbbm{1}_{(t^n_i,t^n_{i+1}]}(t)\mathrm{d}u\nu(\mathrm{d}y)\mathrm{d}z\nonumber\\
    &=  \int_{\| y\|_2\leq 1}\int_0^1 \int_0^1\sum_{i = 0}^{k_n-1}u\langle \nabla \tilde{F}^n_i(x+urRy), Ry\rangle\mathbbm{1}_{(t^n_i,t^n_{i+1}]}(t) (Ry)^j\mathrm{d}r\mathrm{d}u\nu(\mathrm{d}y).\nonumber
\end{align}

\noindent Similarly, with $F$ instead of $F^n$ define $H^j$, and obtain:
\begin{align}\label{jump_term}
    \vert H_n^j(t,x)-H^j(t,x)\vert &\leq \int_{\| y\|_2 < 1} \int_0^1\int_0^1\sum_{i = 0}^{k_n-1}u\|(\nabla F(t^n_{i+1},(X^n_{\wedge t^{n\,-}_{i+1}})^{\Sigma^{1/2}(x+ruRy)-X(t^n_i)}))^T\Sigma^{1/2}\nonumber\\
    &-(\nabla F(t,X_{\wedge t^-}^{\Sigma^{1/2}(x+ruRy)-X(t^-)}))^T\Sigma^{1/2}\|_2\mathbbm{1}_{(t^n_i,t^n_{i+1}]}(t)\|Ry\|_2^2\,\mathrm{d}u\mathrm{d}s\nu(\mathrm{d}y).\nonumber\\
\end{align}

As previously, the gradients in \eqref{jump_term} can be assumed to be bounded using stopping times. Then, the difference of gradients in \eqref{jump_term} is bounded by $C$, and for $|x|\leq M$, $\| H_n(t,x)-H(t,x)\|_2 < C\| R\|_{HS}\int_{\| y_2\| \leq 1} \| y\|_2^2\nu(dy)$, where $\| \cdot \|_{HS}$ is the Hilbert-Schmidt norm. Therefore, since $\| F\|_{L} \leq C\| F\|_\infty$, the convergence of this integral is obtained as in the previous case.\\

\noindent For the last term,

\begin{align}
    \sum_{i = 0}^{k_n-1} \int_{\| y\|_2 < 1}&\Bigg(F^n_i(X(t^-)+(I_d-Q)y)-F^n_i(X(t^-))-\langle \nabla F_i^n(X^n(t^-)),(I_d-Q)y\rangle\Bigg)\,\nu(\mathrm{d}y)\mathbbm{1}_{(t^n_i,t^n_{i+1}]}(t)\nonumber&\\
    = \sum_{i = 0}^{k_n-1} \int_{\| y\|_2 < 1}&\int_0^1 \Bigg(\langle\nabla F^n_i(X(t^-)+s(I_d-Q)y),(I_d-Q)y\rangle-\langle \nabla F_i^n(X^n(t^-)),(I_d-Q)y\rangle\Bigg)\mathrm{d}s\mathbbm{1}_{(t_i,t_{i+1}]}(t)\,\nu(\mathrm{d}y).\nonumber&
\end{align}
\noindent Since $X^n_t$ can be assumed to be bounded, $\nabla F$ is continuous in time, and $X^n(t) \to X(t^-)$ a.e. in $[0,T]$, then by the dominated convergence theorem this last integral converges to:
\begin{align}
    \int_0^T\int_{\| y\|_2 < 1}\int_0^1 \Bigg(&\langle\nabla F(t,X_{t^-}^{s(I_d-Q)y}),(I_d-Q)y\rangle-\langle \nabla F(X^n_{t^-}),(I_d-Q)y\rangle\Bigg)\mathrm{d}s\mathbbm{1}_{(t^n_i,t^n_{i+1}]}(t)\,\nu(\mathrm{d}y)\mathrm{d}t.\nonumber
\end{align}

This convergence allows us to obtain all the terms in \eqref{Optimal_Ito_Formula}, giving the result when $X$ and $B$ are 
truncated. The general case is obtained since if $T_M := \inf\{t > 0: \| X(t)\|_2 \vee \| B(t)\|_2 > M\}$, the theorem holds locally for $X(t\wedge T_M)$, and thus for $X(t)$, by taking $M \to \infty$.
\end{proof}

\begin{remark}\label{Remark:Optimal}
    The representation \eqref{Optimal_Ito_Formula} allows to identify 
    \begin{align*}
        F(t, X_{\wedge t}) - F(0, X_{\wedge 0}) &- \int_0^t DF(s, X_{\wedge s})\mathrm{d}s - \int_0^t \langle \nabla F(s, X_{\wedge s}),\mu\rangle\mathrm{d}s
        -\int_0^t \nabla F(s,X_{\wedge s^-})^T \Sigma^{1/2} \mathrm{d}B(s) \\  &- \int_0^t \int_{\|y\| < 1} \left(F(s,X_{\wedge s^-}^y) - F(s, X_{\wedge s^-})\right)\tilde{N}(ds, dy),
    \end{align*}
    
    \noindent with an expression with minimal regularity conditions on $F$. This notion of optimality was introduced in \cite{Eisenbaum09} for classical functions of univariate L\'{e}vy processes. Theorem \ref{Optimal_Functional} extends this result to the case of functionals evaluated on multivariate L\'{e}vy processes.
\end{remark}

\begin{remark}
(i) If $\Sigma$ is a $d\times d$ invertible matrix, then $Q = I$, and the condition $\int_{\| y\|_2 < 1} \| (I_d-Q)y\|_2 \nu(\mathrm{dy}) < \infty$ is immediately satisfied. In this case, \eqref{Optimal_Ito_Formula} turns into,
\begin{align*}
    F(t,X_{\wedge t})&-F(0,X_{\wedge 0}) = \int_0^t DF(s,X_{\wedge s})\,\mathrm{d}s+\int_0^t \langle \nabla F(s,X_{\wedge s}),\mu\rangle\,\mathrm{d}s
+\int_0^t \nabla F(s,X_{\wedge s^-})^T\Sigma^{1/2}\,\mathrm{d}B(s) 
\nonumber\\
&+ \sum_{s \leq t} (F(s,X_{\wedge s})-F(s,X_{\wedge s^-}))\mathbbm{1}_{\{\|\Delta X(s)\|_2 \geq 1\}}
+\int_0^t \int_{\|y\|_2 < 1} \left(F(s,X_{\wedge s^-}^y)-F(s,X_{\wedge s^-})\right)\,\tilde{N}(\mathrm{d}s,\mathrm{d}y) \nonumber\\
&\qquad \qquad \quad -\sum_{j = 1}^m \int_0^t\int_{\mathbb{R}}\mathcal{A}_j\mathcal{I}_jF(s,X_{\wedge s^-}^{(\Sigma^{1/2}e_j) x-X(s^-)})\,\mathrm{d}L_s^x(B^j).
\end{align*}
\noindent
(ii) The decomposition of functionals of weak Dirichlet processes presented in \cite{bouchardOctober} shows that if $F \in \mathscr{C}^{0,1}([0,T],\mathbb{R}^d)$, then under an additional hypothesis, $F(t,X_{\wedge t})$ can be written as the sum of a local martingale and of an orthogonal process. Theorem \ref{Optimal_Functional} shows that, when $(X(t))_{t\in [0,T]}$ is a L\'{e}vy process, and when the weak derivatives are in $\mathscr{C}^{0,0}([0,T],\mathbb{R}^d)$, one obtains an explicit representation of the orthogonal component.
\end{remark}

\begin{remark}
Let us explain how the L\'{e}vy case, without the Brownian component, recovers the Brownian case when $d_D$ is the complete metric associated to the Skorokhod space. For $F \in \mathscr{C}^{1,2}([0,T],\mathbb{R}^d)$, $F(t,\cdot)$ is continuous in said space, and thus given a sequence of processes $(X^n)_{n \geq 1}$ such that $X^n \xrightarrow[n \to \infty]{\mathcal{L}}X$, where $\mathcal{L}$ denotes convergence in law as elements of the Skorokhod space, it follows that
\begin{align*}
    \lim_{n \to \infty}\mathbb{E}(F(t,X^n_{\wedge t})) = \mathbb{E}(F(t,X)).
\end{align*}
Next, following the construction in \cite[Theorem 2.5]{GaussianApprox}, take a measurable family $\{\mu(\cdot|u): u\in S^{d-1}\}$ of L\'{e}vy measures on $(0,+\infty)$, and a finite positive measure $\lambda$ on the Euclidean unit sphere $S^{d-1}$ whose support is not contained in the intersection of the sphere with any hyperplane, such that together they satisfy the condition,
\begin{align*}
    \lim_{\epsilon \to 0^+} \frac{1}{\epsilon^2}\int_0^\epsilon r^2\mu(\mathrm{d}r|u) = \infty,\, \lambda-\text{a.e}.
\end{align*}

\noindent Then, the L\'{e}vy measure $\tilde{\nu}_\epsilon$ defined via:
\begin{align*}
    \tilde{\nu}_\epsilon(\mathrm{d}r,\mathrm{d}u) &:= \mathbbm{1}_{\{r<\epsilon\}}\mu(\mathrm{d}r|u)\lambda(du),\, r>0, u \in S^{d-1},
\end{align*}
\noindent fulfills the conditions of \cite[Theorem 2.2]{GaussianApprox}. Next, let
\begin{align*}
    \Sigma_\epsilon = \int_{\mathbb{R}^d\setminus\{0\}} xx^T\tilde{\nu}(\mathrm{d}x),\,\,\, \tilde{b}_\epsilon = -\int_{\| \Sigma_\epsilon^{1/2}x\|_2\geq 1}\Sigma_\epsilon^{-1/2}x\tilde{\nu}_\epsilon(\mathrm{d}x),
\end{align*}
 \noindent where the existence of $\Sigma^{1/2}$ and $\Sigma^{-1/2}$ is given by \cite[Lemma 2.1]{GaussianApprox}, and the fact that $\lambda$ is not supported on any hyperplane. Then, the L\'{e}vy processes $(\tilde{X}^\epsilon(t))_{t \in [0,T]}$ with characteristic triplet $(\tilde{b}_\epsilon,0,\tilde{\nu}_\epsilon)$ are such that $X^\epsilon := \Sigma_\epsilon^{-1/2} \tilde{X}^\epsilon \xrightarrow[\epsilon \to 0^+]{\mathcal{L}} B$, where $B$ is a multivariate standard Brownian motion, and the L\'{e}vy process $(X^\epsilon(t))_{t \in [0,T]}$ has triplet $(b_\epsilon,0,\nu_\epsilon)$. The fact that $ X^\epsilon \xrightarrow[\epsilon \to 0^+]{\mathcal{L}} B$, is now used to show that for any $F \in \mathscr{C}^{1,2}([0,T],\mathbb{R}^d)$, $F(T,B_{\wedge T})$ has the same distribution as the one given by the functional It\^{o} formula. Indeed, without loss of generality, assume that $F$ and all its derivatives are bounded and, since then $\lim_{\epsilon \to 0^+} \int_0^T \mathbb{E}(DF(t,X^\epsilon_{\wedge t}))\,\mathrm{d}t = \int_0^T \mathbb{E}(DF(t,B_{\wedge t}))\,\mathrm{d}t$, assume further that $DF = 0$, then,
\begin{align*}
    \mathbb{E}\left(F(T,X^\epsilon_{\wedge T})\right)-\mathbb{E}(F(0,X^\epsilon_{\wedge 0})) &= \int_0^T\int_{\mathbb{R}^d\setminus\{0\}} \mathbb{E}\Bigg(F(t,X^{\epsilon, u}_{\wedge t})-F(t,X^\epsilon_{\wedge t^-})-\langle \nabla F(t,X^\epsilon_{\wedge t}), \mu\rangle\Bigg)\,\nu_\epsilon(\mathrm{d}u)\mathrm{d}t\nonumber\\
    &= \int_0^T \int_{\mathbb{R}^d\setminus\{0\}} \mathbb{E}\left(\int_0^1 Tr(\nabla^2 F(t,X^{\epsilon, su}_{\wedge t^-})uu^t)(1-s)\,\mathrm{d}s\right)\nu_\epsilon(\mathrm{d}u)\mathrm{d}t.
\end{align*}

By \cite[Appendix 1]{Dupire09}, the second derivatives $\partial_{i,j} F$ are uniformly continuous, moreover, $\int_{\mathbb{R}^d\setminus\{0\}} u u^t \,\nu_\epsilon(\mathrm{d}u) = I_d$, and assume the derivatives $\partial_{i,j} F$ are uniformly bounded. Then, proceeding as in Remark 2.2 (i) in \cite{Stein},
\begin{align}
    \Big\vert\int_0^T \int_{\mathbb{R}^d\setminus\{0\}} \int_0^1& Tr(\nabla^2 F(t,X^{su}_{\wedge t^-})uu^t)(1-s)\,\mathrm{d}s\nu_\epsilon(\mathrm{d}u)\mathrm{d}t-\frac{1}{2}\int_0^T \int_{\mathbb{R}^d\setminus\{0\}} Tr(\nabla^2 F(t,X_{\wedge t^-})uu^t)\,\nu_\epsilon(\mathrm{d}u)\mathrm{d}t\Big\vert
    \nonumber\longrightarrow 0,\nonumber
\end{align}
Then, since $X^\epsilon \xrightarrow[\epsilon \to 0^+]{\mathcal{L}} B$,
\begin{align*}\label{stable_approx}
    \lim_{\epsilon \to 0^+} \mathbb{E}(F(T,X^\epsilon_{\wedge T})) &= \mathbb{E}(F(0,0))+\int_0^T \mathbb{E}(Tr(\nabla^2F(t,B_{\wedge t})))\,\mathrm{d}t.
\end{align*}
\noindent Next, let $f:\mathbb{R}\to\mathbb{R}$ be a bounded, infinitely differentiable function, with bounded derivatives and define $(V(t))_{t \in [0,T]}$ via,
\begin{align*}
    V(t) := F(0,0)+\int_0^t \frac{1}{2}Tr(\nabla^2 F(s,B_{\wedge s}))\,\mathrm{d}s+\int_0^t \nabla F(s,B_{\wedge s})\cdot \mathrm{d}B(s).
\end{align*}
\noindent Then, a direct application of the classical It\^{o}'s formula shows that for any such $f$, $\mathbb{E}(f(V))$ is equal to the right-hand side of $\mathbb{E}(F(T,B_{\wedge T}))$ and therefore, $V(T)$ has the same distribution as $F(T,B_{\wedge T})$. 
Finally, by applying the Skorokhod representation theorem, there exist random variables $Y^\epsilon$ with the same distribution as $F(T,X^\epsilon_{\wedge T})$ converging a.s.\@ to a random variable having the same distribution as $Z$.
\end{remark}

\printbibliography
\end{document}